\pdfoutput=1
\documentclass[reqno,oneside,letterpaper]{amsart}
\usepackage{mathrsfs}
\usepackage{amssymb}
\usepackage{wasysym}

\usepackage[utf8]{inputenc}
\usepackage{amsmath, amssymb,bm, cases, mathtools, thmtools}
\usepackage{verbatim}
\usepackage{graphicx}\graphicspath{{figures/}}
\usepackage{multicol}
\usepackage{tabularx}
\usepackage[usenames,dvipsnames]{xcolor}
\usepackage{mathrsfs}
\usepackage{url}
\usepackage[normalem]{ulem}
\usepackage{xstring}
\usepackage{enumitem}

\usepackage[%
    minnames=1,maxnames=99,maxcitenames=3,
    style=alphabetic,
    doi=false,url=false,
    firstinits=true,hyperref,natbib,backend=bibtex]{biblatex}
\renewbibmacro{in:}{%
  \ifentrytype{article}{}{\printtext{\bibstring{in}\intitlepunct}}}
\bibliography{biblio}

\usepackage[colorlinks,citecolor=blue,urlcolor=blue,linkcolor=RawSienna]{hyperref}
\usepackage{hypernat}
\usepackage{datetime}

\DeclareMathAlphabet\EuRoman{U}{eur}{m}{n}
\SetMathAlphabet\EuRoman{bold}{U}{eur}{b}{n}

\usepackage{euscript,microtype}

\usepackage[capitalize]{cleveref}

\crefname{assumption}{Assumption}{Assumptions}
\crefname{claim}{Claim}{Claims}

\makeatletter
\let\reftagform@=\tagform@
\def\tagform@#1{\maketag@@@{\ignorespaces\textcolor{gray}{(#1)}\unskip\@@italiccorr}}
\renewcommand{\eqref}[1]{\textup{\reftagform@{\ref{#1}}}}
\makeatother

\definecolor{WowColor}{rgb}{.75,0,.75}
\definecolor{SubtleColor}{rgb}{0,0,.50}

\newcounter{margincounter}

\declaretheorem[style=plain,numberwithin=section,name=Theorem]{theorem}
\declaretheorem[style=plain,sibling=theorem,name=Lemma]{lemma}

\declaretheorem[style=plain,sibling=theorem,name=Claim]{claim}

\declaretheorem[style=definition,sibling=theorem,name=Definition]{definition}

\declaretheorem[style=definition,sibling=theorem,name=Example]{example}

\crefformat{conditionINNER}{#2(#1)#3}
\crefmultiformat{conditionINNER}
  {(#2#1#3)}
  { and~(#2#1#3)}
  {, (#2#1#3)}
  { and~(#2#1#3)}
\Crefformat{conditionINNER}{Condition~#2(#1)#3}
\Crefmultiformat{conditionINNER}
  {Conditions~(#2#1#3)}
  { and~(#2#1#3)}
  {, (#2#1#3)}
  { and~(#2#1#3)}

\declaretheoremstyle[
    spaceabove=-6pt,
    spacebelow=6pt,
    headfont=\normalfont\bfseries,
    bodyfont = \normalfont,
    postheadspace=1em,
    qed=$\square$,
    headpunct={{}}]{myproofstyle}

\numberwithin{equation}{section}
\numberwithin{theorem}{section}

\renewcommand{\appendixname}{}

\usepackage{amssymb}
\usepackage{mathrsfs}
\usepackage{leftidx}

\usepackage{setspace}

\def\[#1\]{\begin{align}#1\end{align}}
\def\*[#1\]{\begin{align*}#1\end{align*}}

\newcommand{\Nats}{\mathbb{N}}

\DeclareMathOperator*{\newlim}{\mathrm{lim}\vphantom{\mathrm{infsup}}}

\DeclareMathOperator*{\newinf}{\mathrm{inf}\vphantom{\mathrm{infsup}}}
\DeclareMathOperator*{\newsup}{\mathrm{sup}\vphantom{\mathrm{infsup}}}
\renewcommand{\lim}{\newlim}

\renewcommand{\inf}{\newinf}
\renewcommand{\sup}{\newsup}

\newcommand{\cF}{\mathcal F}
\newcommand{\cG}{\mathcal G}

\newcommand{\NSE}[1]{{^{*}#1}}
\newcommand{\ST}{\mathsf{st}}

\newcommand{\cA}{\mathcal{A}}

\newcommand{\cC}{\mathcal{C}}
\newcommand{\cE}{\mathbb{E}}

\newtheorem{open problem}{Open Problem}

\newcommand{\Loeb}[1]{\overline{#1}}

\newcommand{\interior}[1]{%
  {\kern0pt#1}^{\mathrm{o}}%
}

\newcommand{\refproof}[1]{See \cref{#1} for \IfSubStr{#1}{,}{proofs}{a proof}. }

\newif\iflongform
\longformtrue

\iflongform

\else

\fi

\makeatletter
\providecommand*{\toclevel@definition}{0}
\providecommand*{\toclevel@theorem}{0}
\providecommand*{\toclevel@lemma}{0}
\makeatother
\usepackage{leftidx}

\title[Loeb Extension and Loeb Equivalence]
{
Loeb Extension and Loeb Equivalence
}

\newcommand{\cM}{\mathcal{M}}
\newcommand{\cN}{\mathcal{N}}
\newcommand{\cH}{\mathcal{H}}
\newcommand{\cI}{\mathcal{I}}
\newcommand{\cJ}{\mathcal{J}}
\newcommand{\cV}{\mathcal{V}}

\newcommand{\cU}{\mathcal{U}}

\newtheorem{question}{Question}

\begin{document}

\author[R.~M.~Anderson]{Robert M.~Anderson}
\address{Robert M.~Anderson, University of California, Berkeley, Department of Economics}

\author[H.~Duanmu]{Haosui Duanmu}
\address{Haosui Duanmu, University of California, Berkeley, Department of Economics}

\author[D.~Schrittesser]{David Schrittesser}
\address{David Schrittesser, University of Vienna, Kurt G{\"o}del Research Center}

\author[W.~Weiss]{William Weiss}
\address{William Weiss, University of Toronto, Department of Mathematics}


\maketitle

\begin{abstract}
In \citep{loebsun}, the authors raise several open problems on Loeb equivalences between various internal probability spaces. We provide counter-examples for the first two open problems. Moreover, we reduce the third open problem to the following question:  Is the internal algebra generated by the union of two Loeb equivalent internal algebras a subset of the Loeb extension of any one of the internal algebra? 

\end{abstract}
\section{Introduction}
The Loeb measure construction (\citep{Loeb75}) has provided many fruitful applications in various areas in mathematics such as probability theory (see \citep{andersonisrael}, \citep{nsweak}, \citep{localtime}, \citep{Keisler87}, \citep{stoll}, \citep{brownfrac}, \citep{hylevy}, \citep{Markovpaper}, \citep{anderson2018mixhit} etc), statistical decision theory (see \citep{nsbayes}), potential theory (\citep{loebpotential}), mathematical physics (see \citep{nsphysics}) and mathematical economics (see \citep{robeco}, \citep{ali74}, \citep{nsexchange}, \citep{alir75}, \citep{oligopoly}, \citep{rashid78}, \citep{emmons84}, \citep{strongcore}, \citep{secondwelfare}, \citep{andersonbook}, \citep{lawlarge}, \citep{purestrategy}, \citep{sun99}, \citep{khansun991}, \citep{khansun992}, \citep{rauh07}, \citep{indmatching}, \citep{anderson08}, \citep{xsun16}, \citep{duffie18}, \citep{stochasticmec} etc). \footnote{\citep{robeco}, \citep{ali74}, \citep{nsexchange}, \citep{alir75} and \citep{oligopoly} predated the Loeb measure construction. They relied on a careful analysis of the close relationship between the discrete and measure-theoretic properties of hyperfinite sets that the Loeb measure construction so perfectly captures. The arguments in these papers led, via a Loeb space argument, to a very general and completely elementary argument given in \citep{core1}.} These applications are supported by the development of mathematical infrastructures such as integration theory (see \citep{Loeb75} and \citep{andersonisrael}), representation of measures (see \citep{anderson87} and \citep{stincomb}) and Fubini theorem (see \citep{Keisler87}). \citet{loebsun} made an important contribution to the infrastructure, but left four open problems. We give complete solutions to the first two open problems as well as a partial solution to the third open problem. 

Given an internal probability space $(\Omega, \cF, \mu)$, its Loeb extension is defined to be the countably additive probability space $(\Omega, \Loeb{\cF}, \Loeb{\mu})$, where $\Loeb{\cF}$ consists of all sets $B\subset \Omega$ such that
\[
\sup\{\ST(\mu(A)): B\supset A\in \cF\}=\inf\{\ST(\mu(C)): B\subset C\in \cF\}
\] 
and $\Loeb{\mu}(B)$ is defined to be the above supremum. \citet{loebsun} introduced the following definition to compare two internal probability spaces. 

\begin{definition}\label{defloebext}
Let $\cM=(\Omega, \cF, \mu)$ and $\cN=(\Omega, \cG, \nu)$ be two internal probability spaces. 
We say $\cN$ \textbf{Loeb extends} $\cM$ if $\Loeb{\cG}\supset \Loeb{\cF}$ and $\Loeb{\nu}$ extends $\Loeb{\mu}$ as a function. 
We say $\cN$ is \textbf{Loeb equivalent to} $\cM$ if $\Loeb{\cF}=\Loeb{\cG}$ and $\Loeb{\nu}=\Loeb{\mu}$.  
\end{definition}

If $\cF\subset \cG$ and $\nu$ extends $\mu$ as a function, then it is clear that $\cN$ Loeb extends $\cM$. 
However, as pointed out by \citet{loebsun}, Loeb extensions and Loeb equivalence are much more difficult when it is not assumed that $\cF\subset \cG$. As a result, \citet{loebsun} have left four open problems on Loeb extensions and Loeb equivalences. 
We give complete/partial solutions to the first three open problems, which we list below. We say an internal probability space 
$(\Omega, \cF, \mu)$ is \textbf{hyperfinite} if $\cF$ is hyperfinite (As in \citep{loebsun}, we do not require $\Omega$ to be hyperfinite). 

\begin{question}\label{loebeqq1}
Suppose $\cN$ Loeb extends $\cM$. Must $\cN$ has an internal subspace that is Loeb equivalent to $\cM$? What if $\cM$ is assumed to be hyperfinite? 
\end{question} 

\begin{question}\label{loebeqq2}
Suppose $\cN$ Loeb extends $\cM$ and $\cN$ is hyperfinite. Must $\cM$ be Loeb equivalent to a hyperfinite probability space? 
\end{question}

\begin{question}\label{loebeqq3}
Suppose $\cM=(\Omega, \cF, \mu)$ is Loeb equivalent to $\cN=(\Omega, \cG, \nu)$, and $\cH$ be the internal algebra generated by $\cF\cup \cG$.
Must there be an internal probability measure $P$ on $\cH$ such that $\cM$ is Loeb equivalent to $(\Omega, \cH, P)$? What if $\cM$ and $\cN$ are assumed to be hyperfinite? 
\end{question}

\section{Counter-examples for the First Two Questions}
We start this section by introducing the following theorem which provides a useful characterization of Loeb extension. The first part of the theorem is cited from {\citep[][Lemma.~4.3]{loebsun}}.
\begin{theorem}\label{loebeqlemma}
Let $\cM=(\Omega, \cF, \mu)$ and $\cN=(\Omega, \cG, \nu)$ be two internal probability spaces. 
\begin{itemize}
\item $\cN$ Loeb extends $\cM$ if and only if for every $B\in \cF$ there exists $C\in \cG$ such that $B\subset C$ and $\nu(C)\approx \mu(B)$.

\item $\cN$ Loeb extends $\cM$ if and only if for every $B\in \cF$ there exists $C\in \cG$ such that $C\subset B$ and $\nu(C)\approx \mu(B)$.
\end{itemize}
\end{theorem}
\begin{proof} 
The first statement is cited directly from {\citep[][Lemma.~4.3]{loebsun}}. 
We give a proof of the second statement for completeness.  
Suppose $\cN$ Loeb extends $\cM$ and pick $B\in \cF$. 
Then there exists $C\in \cG$ such that $\Omega\setminus B\subset C$ and $\nu(C)\approx \mu(\Omega\setminus B)$. 
Thus, we have $\Omega\setminus C\subset B$ and $\nu(\Omega\setminus C)\approx \mu(B)$. 
Conversely, for every $B\in \cF$, there exists $C\in \cG$ such that $C\subset \Omega\setminus B$ and $\nu(C)\approx \mu(\Omega\setminus B)$. 
Hence, we have $B\subset \Omega\setminus C$ and $\nu(\Omega\setminus C)\approx \mu(B)$. 
By the first statement, we know that $\cN$ Loeb extends $\cM$. 
\end{proof} 

Suppose $\cN$ Loeb extends $\cM$. The following theorem gives a necessary condition that $\cN$ contains an internal subspace which is Loeb equivalent to $\cM$. 

\begin{theorem}\label{loebnecessary}
Suppose $\cN=(\Omega, \cG, \nu)$ has an internal subspace that is Loeb equivalent to $\cM=(\Omega, \cF, \mu)$. 
Then, for every $A\in \cF$, there exists $A'\in \cF$ such that $A'\subset B\subset A$ for some $B\in \cG$ and $\mu(A')\approx \mu(A)$. 
If $\cN$ is hyperfinite, then $B$ can be taken to be $\bigcup\{C\in \cG: C\subset A\}$. 
\end{theorem} 
\begin{proof} 
Let $\cN'=(\Omega, \cG', \nu')$, where $\nu'$ is the restriction of $\nu$ to $\cG'$, be an internal subspace of $\cN$ that is Loeb equivalent to $\cM$. Pick $A\in \cF$.
By \cref{loebeqlemma}, there exists $B\in \cG$ such that $B\subset A$ and $\nu(B)\approx \mu(A)$. 
By \cref{loebeqlemma} again, we know that there exists $A'\in \cF$ such that $A'\subset B$ and $\mu(A')\approx \nu(B)$. 
\end{proof} 

\cref{loebeqq1} asks whether the converse of \cref{loebnecessary} is true. We provide a counter-example below. 

%

\begin{example}\label{example1}
Let $N=\frac{1}{K!}$ for some $K\in \NSE{\Nats}\setminus \Nats$ and let $\Omega=\{\frac{1}{N}, \frac{2}{N}, \dotsc, 1\}$.
Then $\Omega$ includes the set of all rational numbers in $[0, 1]$ as a subset. 
Let $P$ denote the uniform hyperfinite probability measure on $\Omega$, that is, $P(\{\omega\})=\frac{1}{N}$ for every $\omega\in \Omega$. 
Let 
\begin{align*}
\cF=\{\emptyset, \Omega, \{\frac{1}{N}, \frac{2}{N},\dotsc, \frac{1}{2}\}, \{\frac{1}{2}+\frac{1}{N}, \dotsc, 1\}\}
\end{align*}
and let $\cG$ be the internal algebra generated by
\begin{align*}
\{\{\frac{1}{N},\frac{2}{N},\dotsc, \frac{1}{2}-\frac{2}{N}\}, \{\frac{1}{2}-\frac{1}{N},\frac{1}{2}, \frac{1}{2}+\frac{1}{N}\}, \{\frac{1}{2}+\frac{2}{N},\dotsc, 1\}\}.
\end{align*}
Let $\mu$ and $\nu$ be restrictions of $P$ on $\cF$ and $\cG$, respectively. Finally, let $\cM=(\Omega, \cF, \mu)$ and $\cN=(\Omega, \cG, \nu)$. By \cref{loebeqlemma}, it is clear that $\cN$ Loeb extends $\cM$. On the other hand, the Loeb $\sigma$-algebra $\Loeb{\cF}$ generated from $\cF$ is the same as $\cF$. Thus, $\Loeb{\cF}$ does not contain any element in $\{\{\frac{1}{N},\frac{2}{N},\dotsc, \frac{1}{2}-\frac{2}{N}\}, \{\frac{1}{2}-\frac{1}{N},\frac{1}{2}, \frac{1}{2}+\frac{1}{N}\}, \{\frac{1}{2}+\frac{2}{N},\dotsc, 1\}\}$, hence $\cM$ is not Loeb equivalent to any internal subset of $\cN$. 
\end{example}

Since the answer to \cref{loebeqq1} is negative, we now turn our attention to \cref{loebeqq2}. 
The following example shows that the answer to \cref{loebeqq2} is also negative. 

%
%

\begin{example}\label{example2}
Let $N$ be an element of $\NSE{\Nats}\setminus \Nats$ and let $\Omega=\NSE{[0,1]}$. 
Let $\lambda$ denote the Lebesgue measure on $[0,1]$. 
Let $\cF$ be the internal algebra generated by $\NSE{[0,\frac{1}{2})}, \NSE{[\frac{1}{2}, 1]}$ and all $\{a\}$ for $a\in \NSE{[\frac{1}{2}-\frac{1}{N}, \frac{1}{2}+\frac{1}{N}]}$. 
Clearly, $\cF$ is not hyperfinite.
Moreover, a subset of $\NSE{[\frac{1}{2}-\frac{1}{N}, \frac{1}{2}+\frac{1}{N}]}$ is an element of $\cF$ if and only if it is hyperfinite. 
Let $\cG$ be the internal algebra generated by $\{\NSE{[0, \frac{1}{2}-\frac{1}{N})}, \NSE{[\frac{1}{2}-\frac{1}{N}, \frac{1}{2}+\frac{1}{N}]}, \NSE{(\frac{1}{2}+\frac{1}{N}, 1]}\}$. 
Let $\mu$ and $\nu$ be the restrictions of $\NSE{\lambda}$ on $\cF$ and $\cG$, respectively. 
Finally, let $\cM=(\Omega, \cF, \mu)$ and $\cN=(\Omega, \cG, \nu)$. 
It is clear that $\cN$ is hyperfinite and, by \cref{loebeqlemma}, $\cN$ Loeb extends $\cM$. 
Moreover, we have the following lemma: 
\begin{lemma}\label{loebhyperfinite}
For every internal set $F\in \Loeb{\cF}$,  $\Loeb{\mu}(F)=0$ if and only if $F$ is hyperfinite. 
\end{lemma} 
\begin{proof}
Pick $F\in \Loeb{\cF}$.
Clearly, if $F$ is hyperfinite, then $\Loeb{\mu}(F)=0$. 
If $\Loeb{\mu}(F)=0$, then there must exist $F'\in \cF$ such that $F\subset F'$ and $\mu(F')\approx 0$. 
By the construction of $\cF$, $F'$ must be hyperfinite.
As $F$ is internal, $F$ must be hyperfinite. 
\end{proof}

We now show that $\cM$ is not Loeb equivalent to any hyperfinite probability space. 
Suppose not. Let $\cN'=(\Omega, \cG', P)$ be a hyperfinite probability space that is Loeb equivalent to $\cM$. 
For every $a\in \NSE{[\frac{1}{2}-\frac{1}{N}, \frac{1}{2}+\frac{1}{N}]}$, by \cref{loebeqlemma}, there exists $A_{a}\in \cG'$ such that $a\in A_{a}$ and $P(A_{a})\approx 0$. 
Pick $n\in \Nats$ and let 
\begin{align*}
\cA_{n}=\{A\in \cG': (P(A)<\frac{1}{n})\wedge (a\in A\ \text{for *infinitely many}\ a\in \NSE{[\frac{1}{2}-\frac{1}{N}, \frac{1}{2}+\frac{1}{N}]})\}
\end{align*}
Note that there are *infinitely many $a\in \NSE{[\frac{1}{2}-\frac{1}{N}, \frac{1}{2}+\frac{1}{N}]}$ and hyperfinitely many $A\in \cG'$ with $P(A)<\frac{1}{n}$. As each $a\in \NSE{[\frac{1}{2}-\frac{1}{N}, \frac{1}{2}+\frac{1}{N}]}$ must be contained in some $A\in \cG'$ such that $P(A)<\frac{1}{n}$, by the transfer of the pigeonhole principle, $\cA_{n}$ is non-empty for every $n\in \Nats$. 
By the saturation principle, there exists an internal $A_0$ such that $P(A_0)\approx 0$ and $A_0$ contains *infinitely many $a\in \NSE{[\frac{1}{2}-\frac{1}{N}, \frac{1}{2}+\frac{1}{N}]}$. 
As $\cN'$ is Loeb equivalent to $\cM$, we know that $\Loeb{\mu}(A_0)=0$. 
This, however, contradicts \cref{loebhyperfinite}, hence we conclude that $\cM$ is not Loeb equivalent to any hyperfinite probability space. 
\end{example}

In summary, both \cref{loebeqq1} and \cref{loebeqq2} have negative answers. 
In general, if $\cN$ Loeb extends $\cM$, it does not imply $\cM$ is equivalent to some subset of $\cN$.

\section{Partial Solution to the Third Question}
In this section, we give a partial answer to \cref{loebeqq3}. In particular, we reduce \cref{loebeqq3} to the following question: for hyperfinite spaces, is the internal algebra generated by the union of two Loeb equivalent internal algebras a subset of the Loeb extension of any one of the internal algebra? 

Let $(\Omega, \cF, \mu)$ be an internal probability space and let $\cG$ be another internal algebra on $\Omega$. The following theorem shows that it is possible to define an internal measure $\nu$ on $(\Omega, \cG)$ such that $(\Omega, \cF, \mu)$ Loeb extends $(\Omega, \cG, \nu)$ if and only if $\cG\subset \Loeb{\cF}$.

\begin{theorem}\label{newext}
Suppose $\cM=(\Omega, \cF, \mu)$ be a hyperfinite probability space and let $\cG$ be a hyperfinite algebra on $\Omega$.
Then $(\Omega, \cF, \mu)$ Loeb extends $(\Omega, \cG, P)$ for some internal probability measure $P$ if and only if $\cG\subset \Loeb{\cF}$.  
\end{theorem}

\begin{proof}
Suppose, for some internal probability measure $P$, $(\Omega, \cF, \mu)$ Loeb extends $(\Omega, \cG, P)$. 
Then we have $\Loeb{\cG}\subset \Loeb{\cF}$ which implies that $\cG\subset \Loeb{\cF}$. 

Now suppose $\cG\subset \Loeb{\cF}$. 
As $\cF$ is hyperfinite, by the transfer principle, there exists an internal subset $\cF_0$ of $\cF$ such that
\begin{enumerate}
\item $\cF_0$ is a $\NSE{}$partition of $\Omega$
\item For every $A\in \cF_0$, if there exists non-empty $E\in \cF$ such that $E\subset A$, then $E=A$.
\end{enumerate}
Then, for any $F\in \cF$, it is a hyperfinite union of elements from $\cF_0$. Thus, $\cF_0$ generates the internal algebra $\cF$. 
Similarly, there exists an internal subset $\cG_0$ of $\cG$ such that
\begin{enumerate}
\item $\cG_0$ is a $\NSE{}$partition of $\Omega$
\item For every $A\in \cG_0$, if there exists non-empty $E\in \cG$ such that $E\subset A$, then $E=A$.
\end{enumerate}
Let $\cU=\{A\cap B: A\in \cF_0, B\in \cG_0, A\cap B\neq \emptyset\}$. 
It is easy to see that $\cU$ forms a $\NSE{}$partition of $\Omega$ and every element in $\cF\cup \cG$ can be written as a hyperfinite union of elements in $\cU$. For $F\in \cF_0$, let $\cU_{F}=\{F\cap B: B\in \cG_0, B\cap F\neq \emptyset\}$. Then $\cU=\bigcup_{F\in \cF_0}\cU_{F}$ and $\cU_{F}$ is hyperfinite for every $F\in \cF_0$. We now define a function $P': \cU\to \NSE{[0,1]}$. For every $U\in \cU$, $U\in \cU_{F_0}$ for exactly one $F_0\in \cF_0$, let $P'(U)=\frac{\mu(F_0)}{|\cU_{F_0}|}$ where $|\cU_{F_0}|$ denotes the internal cardinality of $\cU_{F_0}$. 
\begin{claim}\label{keyclaim}
For every $A\in \cF$, $\mu(A)=\sum_{U\in \cU, U\subset A}P'(U)$.
\end{claim}
\begin{proof} 
Pick $A\in \cF$. 
Let $\cF_{A}=\{F\in \cF_0: F\subset A\}$. 
For every $U\in \cU$ such that $U\subset A$, it is an element of exactly one element in $\cF_{A}$. 
Moreover, as $\cU$ forms a $\NSE{}$partition of $\Omega$, an element in $\cU$ is either a subset of $A$ or disjoint from $A$. 
Thus, we have 
\begin{align*}
\mu(A)=\sum_{F\in \cF_{A}}\mu(F)=\sum_{F\in \cF_{A}}\sum_{U\in \cU, U\subset F}\frac{\mu(F)}{|\cU_{F}|}=\sum_{U\in \cU, U\subset A}P'(U).
\end{align*}
\end{proof}

Define $P: \cG\to \NSE{[0,1]}$ by letting $P(G)=\sum_{U\in \cU, U\subset G}P'(U)$. 
\begin{claim} 
$P$ is an internal probability measure on $(\Omega, \cG)$.
\end{claim}
\begin{proof}
Clearly we have $P(\emptyset)=0$ and $P(\Omega)=1$. 
Let $G_1, G_2\in \cG$ be two disjoint sets. 
Let $U_0\subset G_1\cup G_2$ be an element of $\cU$. 
As $\cU$ forms a $\NSE{}$partition of $\Omega$ and both $G_1$ and $G_2$ can be written as a hyperfinite union of elements in $\cU$, 
we can conclude that $U_0$ is either a subset of $G_1$ or a subset of $G_2$. 
Thus, we have 
\[
P(G_1\cup G_2)&=\sum_{U\in \cU, U\subset G_1\cup G_2}P'(U)\\
&=\sum_{V\in \cU, V\subset G_1}P'(V)+\sum_{E\in \cU, E\subset G_2}P'(E)\\
&=P(G_1)+P(G_2). 
\]
\end{proof}

We now show that $(\Omega, \cF, \mu)$ Loeb extends $(\Omega, \cG, P)$. 
Pick $G\in \cG$. Let $G_i=\bigcup\{F\in \cF: F\subset G\}$ and let $G_o=\bigcap\{F\in \cF: G\subset F\}$. 
As $\cF$ is hyperfinite, both $G_i$ and $G_o$ are elements of $\cF$. 
Moreover, as $\cG\subset \Loeb{\cF}$, we have $\mu(G_i)\approx \mu(G_o)$, which implies that $\Loeb{\mu}(G)=\ST(\mu(G_i))$. 
As $G_i\subset G\subset G_o$, we have $\sum_{U\in \cU, U\subset G_i}P'(U)\leq P(G)\leq \sum_{U\in \cU, U\subset G_o}P'(U)$.
By \cref{keyclaim}, we have $\sum_{U\in \cU, U\subset G_i}P'(U)=\mu(G_i)$ and $\sum_{U\in \cU, U\subset G_o}P'(U)=\mu(G_o)$.
Thus, we can conclude that $\Loeb{P}(G)=\ST(P(G))=\ST(\mu(G_i))=\Loeb{\mu}(G)$, completing the proof. 
\end{proof}

The following theorem gives a partial answer to \cref{loebeqq3}. 

\begin{theorem}\label{unioneq}
Let $(\Omega, \cF, \mu)$ be a hyperfinite probability space and let $\cG$ be a hyperfinite algebra on $\Omega$. 
Let $\cH$ be the internal algebra generated by $\cF\cup \cG$. 
Then $(\Omega, \cH, P)$ is Loeb equivalent to $(\Omega, \cF, \mu)$ for some internal probability measure $P$ if and only if $\cH\subset \Loeb{\cF}$. 
\end{theorem}
\begin{proof} 
Suppose there exists an internal probability measure $P$ such that $(\Omega, \cH, P)$ is Loeb equivalent to $(\Omega, \cF, \mu)$. 
Then we have $\Loeb{\cH}=\Loeb{\cF}$ which implies that $\cH\subset \Loeb{\cF}$. 

Now suppose $\cH\subset \Loeb{\cF}$. 
By \cref{newext}, there exists an internal probability measure $P$ on $(\Omega, \cH)$ such that $(\Omega, \cF, \mu)$ Loeb extends 
$(\Omega, \cH, P)$. Thus, we have $P(F)\approx \mu(F)$ for every $F\in \cF$. By \cref{loebeqlemma}, $(\Omega, \cH, P)$ Loeb extends $(\Omega, \cF, \mu)$ and we have the desired result. 
\end{proof}

It is natural to ask if \cref{unioneq} remains valid without the hyperfinite assumption. 
\begin{open problem}
Let $\cM=(\Omega, \cF, \mu)$ and $\cN=(\Omega, \cG, \nu)$ be two internal probability spaces that are not hyperfinite.
Let $\cH$ be the internal algebra generated by $\cF\cup \cG$.
Suppose $\cM$ is Loeb equivalent to $\cN$, and $\cH\subset \Loeb{\cF}$.
Must there be an internal probability measure $P$ on $\cH$ such that $\cM$ is Loeb equivalent to $(\Omega, \cH, P)$?
\end{open problem}

\section{Loeb Measurability of $\cH$}
We have shown in previous section that \cref{loebeqq3} is equivalent to the following fundamental problem in Loeb measure theory:
\begin{open problem}\label{posloebquestion}
Let $(\Omega, \cF, \mu)$ be a hyperfinite probability space and let $\cG$ be another hyperfinite algebra on $\Omega$ such that $\cG\subset \Loeb{\cF}$. Let $\cH$ be the hyperfinite algebra generated by $\cF\cup \cG$. Is $\cH\subset \Loeb{\cF}$? 
\end{open problem}
We suspect the answer to \cref{posloebquestion} is negative. However, in this section, we provide a few sufficient conditions that lead to positive answer to \cref{posloebquestion}. 
Let $\cM=(\Omega, \cF, \mu)$ and $\cN=\{\Omega, \cG, \nu\}$ be two hyperfinite probability spaces.
We use the same notation as in the proof of \cref{newext} 
Throughout this section, let $\cF_0$ be an internal subset of $\cF$ such that 
\begin{enumerate}
\item $\cF_0$ is a $\NSE{}$partition of $\Omega$
\item For every $A\in \cF_0$, if there exists non-empty $E\in \cF$ such that $E\subset A$, then $E=A$.
\end{enumerate}
Similarly, let $\cG_0$ be an internal subset of $\cG$ such that
\begin{enumerate}
\item $\cG_0$ is a $\NSE{}$partition of $\Omega$
\item For every $A\in \cG_0$, if there exists non-empty $E\in \cG$ such that $E\subset A$, then $E=A$.
\end{enumerate}
Let $\cU=\{A\cap B: A\in \cF_0, B\in \cG_0, A\cap B\neq \emptyset\}$. Then $\cU$ forms a $\NSE{}$partition of $\Omega$ and every element in $\cH$ can be written as a hyperfinite union of elements in $\cU$. It follows from the Loeb measurability of $\cG$ that $\cU\subset \Loeb{\cF}$. However, it is not clear whether all hyperfinite (not finite) unions of elements in $\cU$ is an element of $\Loeb{\cF}$. Hence it is not straightforward to determine the Loeb measurability of $\cH$. Finally, let $\cF_0'=\{F\in \cF_0: \text{$F$ intersects at least two elements in $\cG_0$}\}$. 

\begin{theorem}\label{immthm}
Suppose $\cG\subset \Loeb{\cF}$ and $\mu(\bigcup \cF_0')\approx 0$. 
Then $\cH\subset \Loeb{\cF}$. 
\end{theorem}
\begin{proof} 
Let $\cF_1=\cF_0\setminus \cF_0'$. 
Pick $H\in \cH$ and, without loss of generality, assume that $H=\bigcup_{i=1}^{K}(A_i\cap B_i)$ where $A_i\in \cF$, $B_i\in \cG$ and $K\in \NSE{\Nats}$.
Let $\cV=\{A_i\cap B_i: i\leq K\}$. 
Then $H=H_1\cup H_2$ where $H_1=\bigcup \{(A_i\cap B_i)\in \cV: A_i\in \cF_0'\}$ and $H_2=\bigcup \{(A_i\cap B_i)\in \cV: A_i\in \cF_1\}$. 
Clearly, $H_1$ is a subset of $\bigcup \cF_0'$. 
As $\mu(\bigcup \cF_0')\approx 0$, by the completeness of Loeb measure, $H_1$ is Loeb measurable and $\Loeb{\mu}(H_1)=0$.
Note that, for every element $F\in \cF_1$, there exists an unique $G\in \cG_0$ such that $F\subset G$. 
Thus, for every $A_i\in \cF_1$, $A_i\cap B_i$ is either $A_i$ or $\emptyset$.
Thus, we know that $H_2\in \cF$.
Hence we conclude that $H\in \Loeb{\cF}$, completing the proof. 
\end{proof} 

We now explore some sufficient conditions that would imply $\mu(\bigcup \cF_0')\approx 0$. We start with the following lemma. 

\begin{lemma}\label{loebseparate}
Let $(\Omega, \cF, \mu)$ be a hyperfinite probability space and let $\cG$ be a hyperfinite algebra on $\Omega$ such that $\cG\subset \Loeb{\cF}$.
Let $\cA$ be an internal subset of $\cF_0'$. 
Suppose there exists $G_1\in \cG$ such that both $G_1$ and $\Omega\setminus G_1$ intersects every element in $\cA$. 
Then $\mu(\bigcup \cA)\approx 0$.  
\end{lemma}
\begin{proof} 
Let $F_{G_1}=\bigcap\{F\in \cF: G_1\subset F\}$. 
Since $G_1$ intersects every element in $\cA$, we know that $\bigcup \cA\subset F_{G_1}$. 
Thus, by the Loeb measurability of $G_1$, we know that $\bigcup \cA\setminus G_1$ has Loeb measure $0$. 
Similarly, we can conclude that $\bigcup \cA\setminus (\Omega\setminus G_1)$ has Loeb measure $0$. 
Thus, we can conclude that $\mu(\bigcup \cA)\approx 0$. 
\end{proof}

We now consider the collection of sets in $\cF$ with infinitesimal measure. 
We show that, if for every $F\in \cF$ with $\mu(F)\approx 0$, there exists some $G\in \cG_0$ such that $F\cap G=\emptyset$, 
then $\mu(\bigcup \cF_0')\approx 0$, which, by \cref{immthm}, implies that $\cH\subset \Loeb{\cF}$ (see \cref{measurecor}).
We start with the following technical lemma. 

\begin{lemma}\label{measureapproach}
Let $(\Omega, \cF, \mu)$ be an internal probability measure and let $\cG$ be an internal algebra on $\Omega$ such that $\cG\subset \Loeb{\cF}$.
Let $\cI_0\subset \cF$ consist of all sets $A$ such that: 
\begin{enumerate}
\item $\mu(A)\approx 0$
\item For all $G\in \cG_0$, $A\cap G\neq \emptyset$. 
\item For all $B\in \cF$ such that $B$ is a proper subset of $A$, there exists some $G\in \cG_0$ such that $B\cap G=\emptyset$.
\end{enumerate}
Then $\bigcup \cI_0$ is $\Loeb{\mu}$-measurable
and $\Loeb{\mu}(\bigcup \cF_0'\setminus \bigcup \cI_0)=0$.
\end{lemma}

\begin{proof}
For $n\in \Nats$, let $\cI_n\subset \cF$ consist of all sets $A$ such that: 
\begin{enumerate}
\item $\mu(A)\leq \frac{1}{n}$
\item For all $G\in \cG_0$, $A\cap G\neq \emptyset$. 
\item For all $B\in \cF$ such that $B$ is a proper subset of $A$, there exists some $G\in \cG_0$ such that $B\cap G=\emptyset$.
\end{enumerate}
By the internal definition principle, $\cI_n$ is internal for every $n\in \Nats$. 
As $\cF$ is hyperfinite, we have $\bigcup \cI_n\in \cF$ for every $n\in \Nats$. 
As $\cI_0=\bigcap_{n\in \Nats} \cI_n$, we conclude that $\bigcup \cI_0=\bigcap_{n\in \Nats}\bigcup \cI_n$ is $\Loeb{\mu}$-measurable. 

Suppose $\Loeb{\mu}(\bigcup \cF_0'\setminus \bigcup \cI_0)>0$.
Let $\cJ_0=\{F\in \cF_0': (\exists A\in \cI_0)(F\subset A)\}$. 
It is easy to see that $\bigcup \cJ_0\subset\bigcup \cI_0$. 
Let $\cE$ be the internal algebra generated by $\cC=\cF_0'\setminus \cJ_0$. 
We immediately have $\Loeb{\mu}(\bigcup \cC)>0$.
We now consider the following algorithm. 
\begin{enumerate}
\item Pick any element $G_1\in \cG_0$ and let $\cC_{G_1}=\{C\in \cC: C\cap G_1\neq \emptyset\}$.

\item For $k\geq 2$, assuming $G_1,\dotsc, G_{k-1}$ have already been chosen, pick $G_k\in \cG_0$ such that 
$G_k\cap \bigcup_{i=1}^{k-1}\bigcup \cC_{G_i}=\emptyset$. 
\end{enumerate}
We continue this process until we reach a natural stopping point. 
Without loss of generality, we assume that this algorithm continue for $K\in \NSE{\Nats}$ many steps 
and consider the set $\bigcup_{i=1}^{K}\bigcup \cC_{G_i}$. 
By construction, for every $A\in \cE$ such that $\mu(A)\approx 0$, there exists some $G\in \cG_0$ such that $B\cap G=\emptyset$.
Hence, we can conclude that $\Loeb{\mu}(\bigcup_{i=1}^{K}\bigcup \cC_{G_i})>0$. 
By construction, both $\bigcup_{i=1}^{K} G_i$ and its complement intersect every element in $\bigcup_{i=1}^{K}\bigcup \cC_{G_i}$. 
By \cref{loebseparate}, this is a contradiction. 
\end{proof}

An immediate consequence of \cref{measureapproach} is:

\begin{lemma}\label{resultmeasure}
Let $(\Omega, \cF, \mu)$ be an internal probability measure and let $\cG$ be an internal algebra on $\Omega$ such that $\cG\subset \Loeb{\cF}$. Let $\cF_0$ and $\cG_0$ be the collections of atoms as defined in the beginning of the section. 
Let $\cI_0$ be the same set as in the statement of \cref{measureapproach}. 
Suppose $\Loeb{\mu}(\bigcup \cI_0)=0$. 
Then $\cH\subset \Loeb{\cF}$. 
\end{lemma}
\begin{proof}
By \cref{measureapproach} and the fact that $\Loeb{\mu}(\bigcup \cI_0)=0$, we can conclude that $\Loeb{\mu}(\bigcup \cF_0')=0$. 
By \cref{immthm}, $\cH\subset \Loeb{\cF}$. 
\end{proof}

We now present the main result of this section, which follows immediately from \cref{resultmeasure}: 

\begin{theorem}\label{measurecor}
Let $(\Omega, \cF, \mu)$ be an internal probability measure and let $\cG$ be an internal algebra on $\Omega$ such that $\cG\subset \Loeb{\cF}$. Let $\cF_0$ and $\cG_0$ be the collections of atoms as defined in the beginning of the section. 
Suppose, for every $F\in \cF$ with $\mu(F)\approx 0$, there exists $G\in \cG_0$ such that $F\cap G=\emptyset$. 
Then $\cH\subset \Loeb{\cF}$. 
\end{theorem}
\begin{proof}
By the assumption of the theorem, we know that $\cI_0$ is empty. 
The result then follows from \cref{resultmeasure}.
\end{proof}

\cref{measurecor} provides a quick way to identify the Loeb measurability of $\cH$. We conclude with the following open question. 

\begin{open problem}
Let $\cM=(\Omega, \cF, \mu)$ and $\cN=(\Omega, \cG, \nu)$ be two hyperfinite probability spaces. 
If $\Loeb{\mu}(\bigcup \cF_0')>0$, is it possible that $\cH\subset \Loeb{\cF}$? 
\end{open problem}

\printbibliography

@article {rauh07,
    AUTHOR = {Rauh, Michael T.},
     TITLE = {Nonstandard foundations of equilibrium search models},
   JOURNAL = {J. Econom. Theory},
  FJOURNAL = {Journal of Economic Theory},
    VOLUME = {132},
      YEAR = {2007},
    NUMBER = {1},
     PAGES = {518--529},
      ISSN = {0022-0531},
   MRCLASS = {91B24 (26E35 90C47)},
  MRNUMBER = {2285619},
MRREVIEWER = {Taras Kudryk},
       DOI = {10.1016/j.jet.2004.07.011},
       URL = {https://doi-org.myaccess.library.utoronto.ca/10.1016/j.jet.2004.07.011},
}

@article {duffie18,
    AUTHOR = {Duffie, Darrell and Qiao, Lei and Sun, Yeneng},
     TITLE = {Dynamic directed random matching},
   JOURNAL = {J. Econom. Theory},
  FJOURNAL = {Journal of Economic Theory},
    VOLUME = {174},
      YEAR = {2018},
     PAGES = {124--183},
      ISSN = {0022-0531},
   MRCLASS = {91B68 (28E05 60J20 91A13 91B40)},
  MRNUMBER = {3759045},
MRREVIEWER = {Xiang Sun},
       DOI = {10.1016/j.jet.2017.11.011},
       URL = {https://doi-org.myaccess.library.utoronto.ca/10.1016/j.jet.2017.11.011},
}

@article {xsun16,
    AUTHOR = {Sun, Xiang},
     TITLE = {Independent random partial matching with general types},
   JOURNAL = {Adv. Math.},
  FJOURNAL = {Advances in Mathematics},
    VOLUME = {286},
      YEAR = {2016},
     PAGES = {683--702},
      ISSN = {0001-8708},
   MRCLASS = {91B68 (03H10 60A10)},
  MRNUMBER = {3415693},
MRREVIEWER = {Krzysztof Szajowski},
       DOI = {10.1016/j.aim.2015.08.030},
       URL = {https://doi-org.myaccess.library.utoronto.ca/10.1016/j.aim.2015.08.030},
}

@article {khansun992,
    AUTHOR = {Ali Khan, M. and Sun, Yeneng},
     TITLE = {Exact arbitrage and portfolio analysis in large asset markets},
   JOURNAL = {Econom. Theory},
  FJOURNAL = {Economic Theory},
    VOLUME = {22},
      YEAR = {2003},
    NUMBER = {3},
     PAGES = {495--528},
      ISSN = {0938-2259},
   MRCLASS = {91B28},
  MRNUMBER = {2003831},
       DOI = {10.1007/s001990200328},
       URL = {https://doi-org.myaccess.library.utoronto.ca/10.1007/s001990200328},
}

@article {khansun991,
    AUTHOR = {Khan, M. Ali and Sun, Yeneng},
     TITLE = {Non-cooperative games on hyperfinite {L}oeb spaces},
   JOURNAL = {J. Math. Econom.},
  FJOURNAL = {Journal of Mathematical Economics},
    VOLUME = {31},
      YEAR = {1999},
    NUMBER = {4},
     PAGES = {455--492},
      ISSN = {0304-4068},
   MRCLASS = {91A10},
  MRNUMBER = {1688400},
MRREVIEWER = {Dan Butnariu},
       DOI = {10.1016/S0304-4068(98)00031-7},
       URL = {https://doi-org.myaccess.library.utoronto.ca/10.1016/S0304-4068(98)00031-7},
}

@article {sun99,
    AUTHOR = {Sun, Yeneng},
     TITLE = {The complete removal of individual uncertainty: multiple
              optimal choices and random exchange economies},
   JOURNAL = {Econom. Theory},
  FJOURNAL = {Economic Theory},
    VOLUME = {14},
      YEAR = {1999},
    NUMBER = {3},
     PAGES = {507--544},
      ISSN = {0938-2259},
   MRCLASS = {91B02 (28B20 28E05 91B50)},
  MRNUMBER = {1725660},
MRREVIEWER = {Erwin Klein},
       DOI = {10.1007/s001990050338},
       URL = {https://doi-org.myaccess.library.utoronto.ca/10.1007/s001990050338},
}

@article {emmons84,
    AUTHOR = {Emmons, David W.},
     TITLE = {Existence of {L}indahl equilibria in measure theoretic
              economies without ordered preferences},
   JOURNAL = {J. Econom. Theory},
  FJOURNAL = {Journal of Economic Theory},
    VOLUME = {34},
      YEAR = {1984},
    NUMBER = {2},
     PAGES = {342--359},
      ISSN = {0022-0531},
   MRCLASS = {90A14 (03H10 28A10 28E05)},
  MRNUMBER = {771008},
MRREVIEWER = {P. A. Loeb},
       DOI = {10.1016/0022-0531(84)90148-0},
       URL = {https://doi-org.myaccess.library.utoronto.ca/10.1016/0022-0531(84)90148-0},
}

@article {rashid78,
    AUTHOR = {Rashid, Salim},
     TITLE = {Existence of equilibrium in infinite economies with
              production},
   JOURNAL = {Econometrica},
  FJOURNAL = {Econometrica. Journal of the Econometric Society},
    VOLUME = {46},
      YEAR = {1978},
    NUMBER = {5},
     PAGES = {1155--1164},
      ISSN = {0012-9682},
   MRCLASS = {90A14},
  MRNUMBER = {508689},
MRREVIEWER = {Truman Bewley},
       DOI = {10.2307/1911440},
       URL = {https://doi-org.myaccess.library.utoronto.ca/10.2307/1911440},
}

@article {stincomb,
    AUTHOR = {Stinchcombe, Maxwell B.},
     TITLE = {Countably additive subjective probabilities},
   JOURNAL = {Rev. Econom. Stud.},
  FJOURNAL = {Review of Economic Studies},
    VOLUME = {64},
      YEAR = {1997},
    NUMBER = {1},
     PAGES = {125--146},
      ISSN = {0034-6527},
   MRCLASS = {60A99 (60A10 62A99 90A10)},
  MRNUMBER = {1433545},
       DOI = {10.2307/2971743},
       URL = {https://doi-org.myaccess.library.utoronto.ca/10.2307/2971743},
}

@article {ali74,
    AUTHOR = {Khan, M. Ali},
     TITLE = {Some remarks on the core of a ``large'' economy},
   JOURNAL = {Econometrica},
  FJOURNAL = {Econometrica. Journal of the Econometric Society},
    VOLUME = {42},
      YEAR = {1974},
     PAGES = {633--642},
      ISSN = {0012-9682},
   MRCLASS = {90A15},
  MRNUMBER = {452542},
       DOI = {10.2307/1913934},
       URL = {https://doi-org.myaccess.library.utoronto.ca/10.2307/1913934},
}

@article {alir75,
    AUTHOR = {Khan, M. Ali and Rashid, S.},
     TITLE = {Nonconvexity and {P}areto optimality in large markets},
   JOURNAL = {Internat. Econom. Rev.},
  FJOURNAL = {International Economic Review},
    VOLUME = {16},
      YEAR = {1975},
     PAGES = {222--245},
      ISSN = {0020-6598},
   MRCLASS = {90A15},
  MRNUMBER = {408768},
MRREVIEWER = {Hildegard Dierker},
       DOI = {10.2307/2525895},
       URL = {https://doi-org.myaccess.library.utoronto.ca/10.2307/2525895},
}

@article {stochasticmec,
    AUTHOR = {Chen, Yi-Chun and He, Wei and Li, Jiangtao and Sun, Yeneng},
     TITLE = {Equivalence of stochastic and deterministic mechanisms},
   JOURNAL = {Econometrica},
  FJOURNAL = {Econometrica. Journal of the Econometric Society},
    VOLUME = {87},
      YEAR = {2019},
    NUMBER = {4},
     PAGES = {1367--1390},
      ISSN = {0012-9682},
   MRCLASS = {91B14},
  MRNUMBER = {3994274},
       DOI = {10.3982/ECTA14698},
       URL = {https://doi-org.myaccess.library.utoronto.ca/10.3982/ECTA14698},
}

@article {indmatching,
    AUTHOR = {Duffie, Darrell and Sun, Yeneng},
     TITLE = {Existence of independent random matching},
   JOURNAL = {Ann. Appl. Probab.},
  FJOURNAL = {The Annals of Applied Probability},
    VOLUME = {17},
      YEAR = {2007},
    NUMBER = {1},
     PAGES = {386--419},
      ISSN = {1050-5164},
   MRCLASS = {60A10 (60J05 91B68 92D20)},
  MRNUMBER = {2292591},
MRREVIEWER = {Christian-Oliver Ewald},
       DOI = {10.1214/105051606000000673},
       URL = {https://doi-org.myaccess.library.utoronto.ca/10.1214/105051606000000673},
}

@article {purestrategy,
    AUTHOR = {Khan, M. Ali. and Rath, Kali P. and Sun, Yeneng},
     TITLE = {On the existence of pure strategy equilibria in games with a
              continuum of players},
   JOURNAL = {J. Econom. Theory},
  FJOURNAL = {Journal of Economic Theory},
    VOLUME = {76},
      YEAR = {1997},
    NUMBER = {1},
     PAGES = {13--46},
      ISSN = {0022-0531},
   MRCLASS = {90D13},
  MRNUMBER = {1477341},
       DOI = {10.1006/jeth.1997.2292},
       URL = {https://doi-org.myaccess.library.utoronto.ca/10.1006/jeth.1997.2292},
}

@article {lawlarge,
    AUTHOR = {Sun, Yeneng},
     TITLE = {Hyperfinite law of large numbers},
   JOURNAL = {Bull. Symbolic Logic},
  FJOURNAL = {The Bulletin of Symbolic Logic},
    VOLUME = {2},
      YEAR = {1996},
    NUMBER = {2},
     PAGES = {189--198},
      ISSN = {1079-8986},
   MRCLASS = {60F15 (03H05)},
  MRNUMBER = {1396854},
       DOI = {10.2307/421109},
       URL = {https://doi-org.myaccess.library.utoronto.ca/10.2307/421109},
}

@article {secondwelfare,
    AUTHOR = {Anderson, Robert M.},
     TITLE = {The second welfare theorem with nonconvex preferences},
      NOTE = {With an appendix by Anderson and Andreu Mas-Colell},
   JOURNAL = {Econometrica},
  FJOURNAL = {Econometrica. Journal of the Econometric Society},
    VOLUME = {56},
      YEAR = {1988},
    NUMBER = {2},
     PAGES = {361--382},
      ISSN = {0012-9682},
   MRCLASS = {90A14 (90A06)},
  MRNUMBER = {935630},
MRREVIEWER = {Jacob Paroush},
       DOI = {10.2307/1911076},
       URL = {https://doi-org.myaccess.library.utoronto.ca/10.2307/1911076},
}

@article {strongcore,
    AUTHOR = {Anderson, Robert M.},
     TITLE = {Strong core theorems with nonconvex preferences},
   JOURNAL = {Econometrica},
  FJOURNAL = {Econometrica. Journal of the Econometric Society},
    VOLUME = {53},
      YEAR = {1985},
    NUMBER = {6},
     PAGES = {1283--1294},
      ISSN = {0012-9682},
   MRCLASS = {90A14 (03H10 90A06)},
  MRNUMBER = {809911},
MRREVIEWER = {Carmen Pagliari},
       DOI = {10.2307/1913208},
       URL = {https://doi-org.myaccess.library.utoronto.ca/10.2307/1913208},
}

@article {nsexchange,
    AUTHOR = {Brown, Donald J. and Robinson, Abraham},
     TITLE = {Nonstandard exchange economies},
   JOURNAL = {Econometrica},
  FJOURNAL = {Econometrica. Journal of the Econometric Society},
    VOLUME = {43},
      YEAR = {1975},
     PAGES = {41--55},
      ISSN = {0012-9682},
   MRCLASS = {90A15},
  MRNUMBER = {443867},
       DOI = {10.2307/1913412},
       URL = {https://doi-org.myaccess.library.utoronto.ca/10.2307/1913412},
}

@article {oligopoly,
    AUTHOR = {Khan, M. Ali},
     TITLE = {Oligopoly in markets with a continuum of traders: an
              asymptotic interpretation},
   JOURNAL = {J. Econom. Theory},
  FJOURNAL = {Journal of Economic Theory},
    VOLUME = {12},
      YEAR = {1976},
    NUMBER = {2},
     PAGES = {273--297},
      ISSN = {0022-0531},
   MRCLASS = {90A15 (02H25)},
  MRNUMBER = {411572},
MRREVIEWER = {Robert C. Davis},
       DOI = {10.1016/0022-0531(76)90078-8},
       URL = {https://doi-org.myaccess.library.utoronto.ca/10.1016/0022-0531(76)90078-8},
}

@article {loebpotential,
    AUTHOR = {Loeb, Peter A.},
     TITLE = {Applications of nonstandard analysis to ideal boundaries in
              potential theory},
   JOURNAL = {Israel J. Math.},
  FJOURNAL = {Israel Journal of Mathematics},
    VOLUME = {25},
      YEAR = {1976},
    NUMBER = {1-2},
     PAGES = {154--187},
      ISSN = {0021-2172},
   MRCLASS = {31D05 (26A98)},
  MRNUMBER = {457757},
MRREVIEWER = {M. G. Shur},
       DOI = {10.1007/BF02756567},
       URL = {https://doi-org.myaccess.library.utoronto.ca/10.1007/BF02756567},
}

@article {hylevy,
    AUTHOR = {Lindstr{\o}m, Tom},
     TITLE = {Hyperfinite {L}\'{e}vy processes},
   JOURNAL = {Stoch. Stoch. Rep.},
  FJOURNAL = {Stochastics and Stochastics Reports},
    VOLUME = {76},
      YEAR = {2004},
    NUMBER = {6},
     PAGES = {517--548},
      ISSN = {1045-1129},
   MRCLASS = {28E05 (60G51)},
  MRNUMBER = {2100020},
MRREVIEWER = {Beloslav Rie\v{c}an},
       DOI = {10.1080/10451120412331315797},
       URL = {https://doi-org.myaccess.library.utoronto.ca/10.1080/10451120412331315797},
}

@article {brownfrac,
    AUTHOR = {Lindstr{\o}m, Tom},
     TITLE = {Brownian motion on nested fractals},
   JOURNAL = {Mem. Amer. Math. Soc.},
  FJOURNAL = {Memoirs of the American Mathematical Society},
    VOLUME = {83},
      YEAR = {1990},
    NUMBER = {420},
     PAGES = {iv+128},
      ISSN = {0065-9266},
   MRCLASS = {60J65 (03H05 03H10 35P20 60J25 60K35 82A41)},
  MRNUMBER = {988082},
MRREVIEWER = {H. Kesten},
       DOI = {10.1090/memo/0420},
       URL = {https://doi-org.myaccess.library.utoronto.ca/10.1090/memo/0420},
}

@article {localtime,
    AUTHOR = {Perkins, Edwin},
     TITLE = {A global intrinsic characterization of {B}rownian local time},
   JOURNAL = {Ann. Probab.},
  FJOURNAL = {The Annals of Probability},
    VOLUME = {9},
      YEAR = {1981},
    NUMBER = {5},
     PAGES = {800--817},
      ISSN = {0091-1798},
   MRCLASS = {60J65 (03H05)},
  MRNUMBER = {628874},
MRREVIEWER = {Barthel W. Huff},
       URL =
              {http://links.jstor.org.myaccess.library.utoronto.ca/sici?sici=0091-1798(198110)9:5<800:AGICOB>2.0.CO;2-D&origin=MSN},
}

@incollection {andersonbook,
    AUTHOR = {Anderson, Robert M.},
     TITLE = {Nonstandard analysis with applications to economics},
 BOOKTITLE = {Handbook of mathematical economics, {V}ol. {IV}},
    SERIES = {Handbooks in Econom.},
    VOLUME = {1},
     PAGES = {2145--2208},
 PUBLISHER = {North-Holland, Amsterdam},
      YEAR = {1991},
   MRCLASS = {90Axx (03H05 03H10 28E05 54J05)},
  MRNUMBER = {1207198},
}

@article {robeco,
    AUTHOR = {Brown, Donald J. and Robinson, Abraham},
     TITLE = {The cores of large standard exchange economies},
   JOURNAL = {J. Econom. Theory},
  FJOURNAL = {Journal of Economic Theory},
    VOLUME = {9},
      YEAR = {1974},
    NUMBER = {3},
     PAGES = {245--254},
      ISSN = {0022-0531},
   MRCLASS = {90A15},
  MRNUMBER = {475731},
MRREVIEWER = {M. Ali Khan},
       DOI = {10.1016/0022-0531(74)90050-7},
       URL = {https://doi-org.myaccess.library.utoronto.ca/10.1016/0022-0531(74)90050-7},
}

@article {core1,
    AUTHOR = {Anderson, Robert M.},
     TITLE = {An elementary core equivalence theorem},
   JOURNAL = {Econometrica},
  FJOURNAL = {Econometrica. Journal of the Econometric Society},
    VOLUME = {46},
      YEAR = {1978},
    NUMBER = {6},
     PAGES = {1483--1487},
      ISSN = {0012-9682},
   MRCLASS = {90A06 (90A16)},
  MRNUMBER = {513701},
MRREVIEWER = {Carl P. Simon},
       DOI = {10.2307/1913840},
       URL = {https://doi-org.myaccess.library.utoronto.ca/10.2307/1913840},
}

@article {anderson08,
    AUTHOR = {Anderson, Robert M. and Raimondo, Roberto C.},
     TITLE = {Equilibrium in continuous-time financial markets: endogenously
              dynamically complete markets},
   JOURNAL = {Econometrica},
  FJOURNAL = {Econometrica. Journal of the Econometric Society},
    VOLUME = {76},
      YEAR = {2008},
    NUMBER = {4},
     PAGES = {841--907},
      ISSN = {0012-9682},
   MRCLASS = {91B28 (60H20 91B50 91B62)},
  MRNUMBER = {2433482},
       DOI = {10.1111/j.1468-0262.2008.00861.x},
       URL = {https://doi-org.myaccess.library.utoronto.ca/10.1111/j.1468-0262.2008.00861.x},
}

@book {nsphysics,
    AUTHOR = {Albeverio, Sergio and H{\o}egh-Krohn, Raphael and Fenstad, Jens
              Erik and Lindstr{\o}m, Tom},
     TITLE = {Nonstandard methods in stochastic analysis and mathematical
              physics},
    SERIES = {Pure and Applied Mathematics},
    VOLUME = {122},
 PUBLISHER = {Academic Press, Inc., Orlando, FL},
      YEAR = {1986},
     PAGES = {xii+514},
      ISBN = {0-12-048860-4; 0-12-048861-2},
   MRCLASS = {03H10 (00A05 03H05 60H99 60J99 81C99 81E25 82-02)},
  MRNUMBER = {859372},
MRREVIEWER = {Frank Wattenberg},
}

@article {stoll,
    AUTHOR = {Stoll, Andreas},
     TITLE = {A nonstandard construction of {L}\'{e}vy {B}rownian motion},
   JOURNAL = {Probab. Theory Relat. Fields},
  FJOURNAL = {Probability Theory and Related Fields},
    VOLUME = {71},
      YEAR = {1986},
    NUMBER = {3},
     PAGES = {321--334},
      ISSN = {0178-8051},
   MRCLASS = {60J65 (03H05 60F17)},
  MRNUMBER = {824706},
MRREVIEWER = {Tom L. Lindstr\"{o}m},
       DOI = {10.1007/BF01000208},
       URL = {https://doi-org.myaccess.library.utoronto.ca/10.1007/BF01000208},
}

@article {loebsun,
    AUTHOR = {Keisler, H. Jerome and Sun, Yeneng},
     TITLE = {A metric on probabilities, and products of {L}oeb spaces},
   JOURNAL = {J. London Math. Soc. (2)},
  FJOURNAL = {Journal of the London Mathematical Society. Second Series},
    VOLUME = {69},
      YEAR = {2004},
    NUMBER = {1},
     PAGES = {258--272},
      ISSN = {0024-6107},
   MRCLASS = {60B10 (03H05 28E05 60A10)},
  MRNUMBER = {2025340},
MRREVIEWER = {D. Butkovi\'{c}},
       DOI = {10.1112/S0024610703004794},
       URL = {https://doi-org.myaccess.library.utoronto.ca/10.1112/S0024610703004794},
}

@article {andersonisrael,
    AUTHOR = {Anderson, Robert M.},
     TITLE = {A non-standard representation for {B}rownian motion and {I}t\^{o}
              integration},
   JOURNAL = {Israel J. Math.},
  FJOURNAL = {Israel Journal of Mathematics},
    VOLUME = {25},
      YEAR = {1976},
    NUMBER = {1-2},
     PAGES = {15--46},
      ISSN = {0021-2172},
   MRCLASS = {60G05 (02H25 60H05)},
  MRNUMBER = {0464380},
MRREVIEWER = {P. A. Loeb},
       DOI = {10.1007/BF02756559},
       URL = {https://doi-org.myaccess.library.utoronto.ca/10.1007/BF02756559},
}

@unpublished{anderson2018mixhit,
  title={Mixing Times and Hitting Times for General {M}arkov Processes},
  author={Anderson, Robert M. and Duanmu, Haosui and Smith, Aaron },
  year={2018},
  Note={submitted}
}

@article {nsweak,
    AUTHOR = {Anderson, Robert M. and Rashid, Salim},
     TITLE = {A nonstandard characterization of weak convergence},
   JOURNAL = {Proc. Amer. Math. Soc.},
  FJOURNAL = {Proceedings of the American Mathematical Society},
    VOLUME = {69},
      YEAR = {1978},
    NUMBER = {2},
     PAGES = {327--332},
      ISSN = {0002-9939},
   MRCLASS = {28A32 (02H25 60B10)},
  MRNUMBER = {0480925},
MRREVIEWER = {P. A. Loeb},
       DOI = {10.2307/2042621},
       URL = {http://dx.doi.org.myaccess.library.utoronto.ca/10.2307/2042621},
}

@unpublished{nsbayes,
         AUTHOR= {Duanmu,Haosui and Roy, Daniel M.},
         TITLE = {On extended admissible procedures and their nonstandard Bayes risk},
         YEAR = {2018},
         NOTE = {The Annals of Statistics, under revision}
         }

@unpublished{Markovpaper,
         AUTHOR= {Duanmu,Haosui and Rosenthal,J.S. and Weiss,William},
         TITLE = {Ergodicity of Markov processes via non-standard analysis},
         YEAR = {2018},
         NOTE = {Memoirs of the American Mathematical Society, to appear}
         }

@article {anderson87,
    AUTHOR = {Anderson, Robert M.},
     TITLE = {Star-finite representations of measure spaces},
   JOURNAL = {Trans. Amer. Math. Soc.},
  FJOURNAL = {Transactions of the American Mathematical Society},
    VOLUME = {271},
      YEAR = {1982},
    NUMBER = {2},
     PAGES = {667--687},
      HIDEISSN = {0002-9947},
     CODEN = {TAMTAM},
   MRCLASS = {03H05 (28D05 60A10)},
  MRNUMBER = {654856},
MRREVIEWER = {Peter W. Day},
       DOI = {10.2307/1998904},
       URL = {http://dx.doi.org.myaccess.library.utoronto.ca/10.2307/1998904},
}

@article {Keisler87,
    AUTHOR = {Keisler, H. Jerome},
     TITLE = {An infinitesimal approach to stochastic analysis},
   JOURNAL = {Mem. Amer. Math. Soc.},
  FJOURNAL = {Memoirs of the American Mathematical Society},
    VOLUME = {48},
      YEAR = {1984},
    NUMBER = {297},
     PAGES = {x+184},
      HIDEISSN = {0065-9266},
     CODEN = {MAMCAU},
   MRCLASS = {60H10 (03H05)},
  MRNUMBER = {732752},
MRREVIEWER = {Tom L. Lindstr{\"o}m},
       DOI = {10.1090/memo/0297},
       URL = {http://dx.doi.org.myaccess.library.utoronto.ca/10.1090/memo/0297},
}

@article {Loeb75,
    AUTHOR = {Loeb, Peter A.},
     TITLE = {Conversion from nonstandard to standard measure spaces and
              applications in probability theory},
   JOURNAL = {Trans. Amer. Math. Soc.},
  FJOURNAL = {Transactions of the American Mathematical Society},
    VOLUME = {211},
      YEAR = {1975},
     PAGES = {113--122},
      HIDEISSN = {0002-9947},
   MRCLASS = {28A10 (02H25 60J99)},
  MRNUMBER = {0390154},
MRREVIEWER = {R. B. Kirk},
}

\end{document}